\documentclass[12pt]{amsart}
\usepackage{ulem}
\usepackage{framed}
\usepackage{mathrsfs}
\usepackage[pagewise]{lineno}

\usepackage{amsthm}
\usepackage{amssymb}
\usepackage{amscd}
\usepackage{amsmath}
\usepackage[all]{xy}
\usepackage{comment}
\usepackage{algorithmic,algorithm}
\usepackage{here}
\usepackage{graphicx}
\usepackage{color} %
\usepackage{enumerate} %
\usepackage{listliketab}
\usepackage{multirow}
\usepackage{enumitem}

\usepackage[T1]{fontenc}

\makeatletter
  
  \@addtoreset{algorithm}{subsection}
\makeatother

\def\Frob{{\rm Frob}}

\def\dim{{\rm dim}}

\newcommand{\GL}{\operatorname{GL}}

\newcommand{\PGL}{\operatorname{PGL}}

\newcommand{\Hom}{\operatorname{Hom}}

\newcommand{\Aut}{\operatorname{Aut}}

\newcommand{\diag}{\operatorname{diag}}

\newtheorem{theorem}{Theorem}[subsection]
\newtheorem{cor}[theorem]{Corollary}
\newtheorem{lem}[theorem]{Lemma}
\newtheorem{prop}[theorem]{Proposition}

\newtheorem{theor}{Theorem}

\theoremstyle{definition}

\newtheorem{rem}[theorem]{Remark}

\newtheorem*{expectt}{Expectation}
\newtheorem{remm}{Remark}[section]

\makeatletter
\def\@seccntformat#1{\csname the#1\endcsname. }
\long\def\@makefntext#1{\parindent 1em\noindent 
\@hangfrom{\hbox to 1.8em{\hss $^{\@thefnmark}$}}#1}
\def\legendre@dash#1#2{\hb@xt@#1{%
  \kern-#2\p@
  \cleaders\hbox{\kern.5\p@
    \vrule\@height.2\p@\@depth.2\p@\@width\p@
    \kern.5\p@}\hfil
  \kern-#2\p@
  }}
\def\@legendre#1#2#3#4#5{\mathopen{}\left(
  \sbox\z@{$\genfrac{}{}{0pt}{#1}{#3#4}{#3#5}$}%
  \dimen@=\wd\z@
  \kern-\p@\vcenter{\box0}\kern-\dimen@\vcenter{\legendre@dash\dimen@{#2}}\kern-\p@
  \right)\mathclose{}}
\newcommand\legendre[2]{\mathchoice
  {\@legendre{0}{1}{}{#1}{#2}}
  {\@legendre{1}{.5}{\vphantom{1}}{#1}{#2}}
  {\@legendre{2}{0}{\vphantom{1}}{#1}{#2}}
  {\@legendre{3}{0}{\vphantom{1}}{#1}{#2}}
}
\def\dlegendre{\@legendre{0}{1}{}}
\def\tlegendre{\@legendre{1}{0.5}{\vphantom{1}}}
\renewcommand\section{\@startsection {section}{1}{\z@}%
 {-3.5ex \@plus -1ex \@minus -.2ex}%
 {2.3ex \@plus.2ex}%
 {\normalfont\large\bfseries}}
\makeatother

\newcommand{\Cyc}{{\rm Z}}

\begin{document}
\baselineskip=17pt

\title[The $a$-numbers of non-hyperelliptic curves of genus 3]{The $a$-numbers of non-hyperelliptic curves of genus 3 with cyclic automorphism group of order 6}

\author[R. Ohashi]{Ryo Ohashi}
\address{Graduate School of Information Science and Technology\\
The University of Tokyo}
\email{ryo-ohashi@g.ecc.u-tokyo.ac.jp}

\author[M. Kudo]{Momonari Kudo}
\address{Department of Information and Communication\\
Faculty of Information Engineering\\
Fukuoka Institute of Technology\\
Wajiro-higashi 3-30-1, Higashi-ku, 811-0295 Fukuoka, Japan}
\email{m-kudo@fit.ac.jp}

\author[S. Harashita]{Shushi Harashita}
\address{Graduate School of Environment and Information Sciences\\
Yokohama National University}
\email{harasita@ynu.ac.jp}



\date{}

\subjclass[2020]{Primary 14H45, 14H50; Secondary 33C05}

\keywords{Algebraic curves, Genus-3 curves, Canonical curves, Cartier-Manin matrix, Hypergeometric series}

\begin{abstract}
In this paper, we study non-hyperelliptic curves of genus $3$ with cyclic automorphism group of order $6$.
{Over an algebraically closed field $K$ of characteristic $\neq 2,3$,} such curves are written as plane quartics $C_r: x^3 z + y^4 + r y^2 z^2 + z^4 = 0$ with one parameter $r$.
As the first main theorem, {we show that $r\neq 0,\pm 2$ and} give a necessary and sufficient condition with respect to $r$ and $r'$ such that $C_r \cong C_{r'}$.
By describing the Hasse-Witt matrix of $C_r$ in terms of a certain Gauss' hypergeometric series,
we obtain the second main theorem, where we determine the possible $a$-number of $C_r$, and give the exact number of isomorphism classes over $K$ of such curves attaining the possible maximal $a$-number.
\end{abstract}

\maketitle

\numberwithin{equation}{section}
\section{Introduction}\label{sec:Intro0}
Let $K$ be an algebraically closed field of positive characteristic $p$.
Throughout this paper, by a curve we mean a non-singular projective variety of dimension one.
A curve $C$ of genus $g$ over $K$ is said to be {\it superspecial} if $J(C) \cong E^g$ for a supersingular elliptic curve $E$, where $J(C)$ denotes the Jacobian variety of $C$.
There are two well-known equivalent conditions for $C$ to be superspecial:
One is that the Frobenius $\Frob_C^{\ast}$ on the first cohomology group $H^1(C,\mathcal{O}_C)$ acts as zero~\cite[Theorem 4.1]{Nygaard}, where $\Frob_C$ is the absolute Frobenius morphism on $C$, and the other is
that the $a$-number $a(C)$ is equal to the genus $g$,
where the $a$-number of $C$ is defined by $a(C):= \dim_K \Hom(\alpha_p, J(C))$ with the group scheme $\alpha_p$ obtained as the kernel of the Frobenius on the additive group ${\mathbb G}_a$.
This equivalence can be generalized as follows:
A matrix $H$ representing $\Frob_C^{\ast}$ with respect to a suitable basis for $H^1(C,\mathcal{O}_C)$ is called the {\it Hasse-Witt matrix} of $C$, and one has
$a(C) = g - \mathrm{rank}(H)$ (see e.g., \cite[\S\;2.4]{H_RIMS20} for an exposition of this fact).
Here, the $a$-number is used for the stratification on the space of curves of given genus $g$.
For this context, determining the $a$-number of a given curve is a very important problem in the study of algebraic curves and their moduli spaces.

This paper focuses on the case where $C$ is of genus $g=3$ and non-hyperelliptic curve. We here recall that Lercier-Ritzenthaler-Rovetta-Sijsling \cite[Theorem 3.1]{LRRS} classified such curves $C$ over $K$ of characteristic $p \geq 5$ into $13$ types by structures of the automorphism group of $C$ (we write ${\rm Aut}(C)$ for this group) as finite groups, and they also gave an explicit defining equation of $C$ for each type.
In the following, we denote by $\Cyc_n$ (resp.\ ${\rm D}_{2n}$) the cyclic\ (resp.\ dihedral) group of order $n$ (resp.\ $2n$) and denote by ${\rm S}_n$ the symmetric group of degree $n$.
Among these 13 types, the families of dimension $1$ are the following three types:
${\rm Aut}(C) \cong \Cyc_6,\,\Cyc_4 \times {\rm D}_4$ or ${\rm S}_4$.
Of these, it is known (cf.\ \cite[Sections 4-5]{Top}) that the Jacobian varieties of the latter two curves are $(2,2,2)$-isogenous to the product of three elliptic curves $E_i$, and hence the $a$-number of $C$ is completely described in terms of those of $E_i$ (therefore we do not treat the two cases in this paper).

From the above reason, in this paper we study non-hyperelliptic curves $C$ of genus $3$ with cyclic automorphism group $\Cyc_6$.
According to \cite[Theorem 3.1]{LRRS}, an explicit defining equation of such a curve is given as
\[
C_r: x^3 z + y^4 + r y^2 z^2 + z^4 = 0 \quad {\rm for}\ r \in K,
\]
and $r^2$ is a complete invariant of the curve (over any algebraically closed field), as shown in the proof of \cite[Theorem 3.3]{LRRS}. In Section \ref{Cr} of this paper,
we study such a description of these curves for the reader's convenience with a different argument from \cite{LRRS}
and we find a condition on $r$ so that $C_r$ is nonsingular and the automorphism group of $C_r$
is isomorphic to $\Cyc_6$.
\begin{theor}\label{thm:main10}
 {Assume that the characteristic of $K$ is not $2$ nor $3$.
Any non-hyperelliptic curve of genus $3$ with automorphism group containing a cyclic group of order $6$ is isomorphic to $C_r$ for an $r\ne \pm 2$, and 
the automorphism group of $C_r$ for $r\ne \pm 2$ is a cyclic group of order $6$
if and only if $r \ne 0$.}
For $r,r' \notin \{0,\pm2\}$, we have $C_r \cong C_{r'}$ if and only if $r^2 = {r'}^2$. Moreover, if $C_r$ is superspecial, then $r^2 \in \mathbb{F}_{p^2}$.
\end{theor}
\noindent On the other hand, it is known (cf.\ Remark \ref{rem:cyclic} below for details) that non-hyperelliptic curves of genus $3$ whose automorphism group contains $\Cyc_6$ can be represented as
\[
    M_{(t_1,t_2,t_3,t_4)}: y^6 = (x-t_1)(x-t_2)^3(x-t_3)^4(x-t_4)^4,
\]
where all $t_i$'s are distinct elements in $K$. In Subsection \ref{Moonen}, we give an invariant for isomorphism classes of these curves by using Theorem A.

Li, Mantovan, Pries and Tang~\cite[Section 6.1]{LMPT1} determined possible Newton polygons of $M_{(t_1,t_2,t_3,t_4)}$, which implies that the possible $a$-numbers of $C_r$ are also determined (see the beginning of Section \ref{sec:C6} below). As a different way from the prior work, we 
compute the Hasse-Witt matrix of $C_r$ explicitly by using a Gauss' hypergeometric series, which gives another proof of the result on the possible $a$-numbers of $C_r$. Our method has an advantage: We can determine the exact number of isomorphism classes of $C_r$ attaining the possible maximal $a$-number:
\begin{theor}\label{thm:main11}
We assume that $p \geq 5$. Then, we have the following:
\begin{enumerate}
    \item[(1)] (cf.\ \cite[Section 6.1]{LMPT1}) The possible $a$-numbers of non-hyperelliptic curves $C_r$ of genus 3 such that $\Aut(C_r)\simeq {\Cyc}_6$ are as follows:
\begin{enumerate}
\item[(i)] If $p\equiv 1 \pmod 6$, we have $a(C_r) = 0$ or $2$, and
\item[(ii)] If  $p\equiv 5 \pmod 6$, we have $a(C_r) = 1$ or $3$.
\end{enumerate}
\item[(2)] In both the cases (i) and (ii), the exact number of isomorphism classes of $C_r\!$ attaining the maximal $a$-number is $\left\lfloor p/12 \right\rfloor$.
\end{enumerate}
\end{theor}

We remark that the enumeration result in Theorem \ref{thm:main11} for $p\equiv 5 \pmod 6$ with $a$-number $=3$ is found in Brock's dissertation \cite[Theorem 3.15, II(e)]{Brock}, where he used the result by Hashimoto \cite{Hashimoto} on the class numbers of quaternion unitary groups. Therefore, the new result is that we obtain the number of isomorphism classes of $C_r$ such that $a(C_r) = 2$ when $p \equiv 1 \pmod{6}$.

The rest of this paper is organized as follows: In Subsection \ref{Cr}, we show that non-hyperelliptic curves of genus 3 with cyclic automorphism group $\Cyc_6$ are written as the form $C_r$ for $r \neq 0,\pm 2$, and we prove Theorem \ref{thm:main10}. In Subsection \ref{Moonen}, we examine the relationship between $C_r$ and $M_{(t_1,t_2,t_3,t_4)}$. At the beginning of Section 3, we review a part of Li-Mantovan-Pries-Tang's works~\cite{LMPT1} briefly.
After that, we compute explicitly the Hasse-Witt matrices of our curves, dividing into two cases depending on whether $p \equiv 5 \pmod{6}$ or $p \equiv 1 \pmod{6}$. Finally, we give a concluding remark in Section 4.

\section{Non-hyperelliptic genus-3 curves with cyclic automorphism group of order 6}\label{sec:Intro}

As described in Introduction, we study non-hyperelliptic curves of genus $3$ with cyclic automorphism groups of order $6$ in this paper. Let $K$ be an algebraically closed field of characteristic $0$ or $p\ge 5$.
In this section, we give a complete description of defining equations of such curves. In Subsection \ref{Cr} below, we study an explicit defining equation of our curves as plane quartics, and we prove Theorem A. In Subsection 2.2, we study another defining equation of our curves as Moonen's 9th special family, and we examine the correspondence between these two forms. 


\subsection{Plane quartic with cyclic automorphism group of order $6$}\label{Cr}
We start with considering the plane quartic given in Case\,(7) of \cite[Theorem 3.1]{LRRS}:
\begin{equation}\label{DefEq:Cr}
	C_r: x^3 z + y^4 + r y^2 z^2 + z^4 = 0, \quad {\rm for}\ r \in K.
\end{equation}
The automorphism groups of $C_r$ contains $\Cyc_6$ for all $r$. Indeed, for a primitive third root $\zeta$ of unity, the automorphism
\[
    \gamma: C_r \rightarrow C_r\ ;\,(x:y:z)\mapsto (\zeta x: -y: z)
\]
is of order $6$. We need to determine when $C_r$ is non-singular,
when $\Aut(C_r)$ is isomorphic to $\Cyc_6$
and when $C_r$ and $C_{r'}$ for $r,r' \in K$ are isomorphic.
\begin{lem}\label{lem:nonsing}
The curve $C_r$ is non-singular if and only if $r \neq \pm 2$.
\end{lem}
\begin{proof}
Using the Jacobian criterion, the only possible singular points are $(x:y:z) \in \mathbb{P}^2$ satisfying the equation $3x^2z = 2y(2y^2+rz^2) = x^3+2z(2z^2+ry^2) = 0$. From this equation and \eqref{DefEq:Cr}, we have $x=0$, $y\ne 0$ and $z\ne 0$. Then, it follows from $2y^2+rz^2=0$ and $2z^2+ry^2=0$ that $(4-r^2)y^2 = 0$ and $(4-r^2)z^2 = 0$. Hence, there exists a solution without $x=y=z=0$ if and only if $r = \pm 2$.
\end{proof}
From now on, we assume $r \neq \pm 2$.
It follows from
\[
    x^3z + y^4 + ry^2z^2 + z^4 = \Bigl(y^2 + \frac{r}{2}z^2\Bigr)^{\!2} + x^3z - \Bigl(\frac{r^2}{4}-1\Bigr)z^4
\]
that we have a degree-2 morphism
\begin{equation}\label{standard elliptic quotient}
    \rho_r: C_r \to {E_r}, \quad (x:y:z) \mapsto 
    \biggl(-\frac{x}{z},\,\frac{y^2 + \frac{r}{2}z^2}{z^2}\biggr) 
    =: (X,Y),
\end{equation}
where {$E_r$} is the elliptic curve defined as
{
\begin{equation}\label{eq:E_0}
E_r: Y^2 = X^3+(r^2/4 - 1).
\end{equation}
Note that the $j$-invariant of $E_r$ is zero.}
This morphism $\rho_r: C_r\to {E_r}$ is just the quotient of $C_r$ by the involution $\gamma^3$ and it is ramified {at
infinity $O$} and the three points {$(\zeta^i,r/2)$ for $i=0,1,2$.} 
Note that the set of {branch} points
is stable under the action of the order-3 group $\langle \gamma\rangle/\langle \gamma^3\rangle$.

The assertion of the following lemma has been obtained by \cite[Theorem 3.1]{LRRS}.
For the reader's convenience, we will review it, which helps the reader understand the structure of a non-hyperelliptic curve of genus $3$ having $\Cyc_6$ 
as a subgroup of its automorphism group.
\begin{lem}\label{C is C_r}
Let $C$ be a non-hyperelliptic curve of genus $3$ such that a subgroup $G$ of $\Aut(C)$ is isomorphic to $\Cyc_6$. Then $C$ is isomorphic to $C_r$
for some $r \neq \pm 2$.
\end{lem}
\begin{proof}
Note that any automorphism of $C$ extends to an automorphism of ${\mathbb P}^2$, as $C$ is canonical. Any nontrivial involution in $\PGL_3(K)$ has a representative in $\GL_3(K)$
which is conjugate to $\diag(-1,1,1)$ since ${\rm char}(K) \ne 2$.
In particular, any involution $\sigma$ is identical on a certain hyperplane $H$
(and sends a vector $v$ to $-v$). 
Hence, the quotient $C \to E$ by $\langle\sigma\rangle$ is ramified at four points (the intersection of quartic $C$ and the hyperplane $H$), which are distinct (otherwise $C$ is singular).
Then the Hurwitz's formula says that $E$ is of genus one.

Let $\sigma$ be the involution in $G$, and let $E$ be the quotient $C/\langle\sigma\rangle$.
Now $\Cyc_3 \cong G/\langle\sigma\rangle$
acts on $E$ non-trivially and moreover {the set of } the {branch} points of the quotient morphism $C \to E$ is stable under the action of $\Cyc_3 \cong G/\langle\sigma\rangle$.
The action of $\Cyc_3$ on the $4$ {branch} points
has to be non-trivial (otherwise $\Cyc_3$ acts on $E$ trivially), and therefore one is a fixed point and the other 3 points make an orbit of $\Cyc_3$.
Consider $E$ as an elliptic curve with
 the fixed point as the origin $O$; then
$E$ is an elliptic curve of $j$-invariant $0$,
 say $E: Y^2 = X^3 + 1$, and
 $C$ is the double cover of $E$ ramified at $O$ and the three points
 \[
    (X,Y)=(\zeta^ia,b), \quad i=0,1,2
\]
for a certain $(a,b)$ satisfying $b^2=a^3+1$. 
We choose $r$ so that $a^3=4/(r^2-4)$, we identify
$E$ and $E_r$  as in \eqref{eq:E_0}
by the isomorphism
$E \to E_r$ sending $(X,Y)$ to $(X/a,rY/2b)$. 
Replacing $b$ by $-b$ by using the hyperelliptic involution $(X,Y)\mapsto (X,-Y)$ in $\Aut(E_0)$ if necessary,
the two double covers $C\to E_r$ and $C_r \to E_r$
have the same branch points.
Hence $C$ is isomorphic to $C_r$.
\end{proof}

{Let $D$ be the genus-1 curve obtained by forgetting the origin from the elliptic curve $Y^2=X^3+1$ with $j$-invariant $0$.}
In the proof of Lemma \ref{C is C_r}, to a cyclic subgroup $G \subset \Aut(C)$ of order $6$
we associate
a degree-2 morphism {$C \to D$,
which is the quotient by the order-2 subgroup of $G$.}
We call such a morphism of degree 2
up to the action by $\Aut({D})$,
a $\Cyc_6$-{\it elliptic quotient of $C$},
{where the action of $\Aut({D})$
is defined by the composition of $C\to D$ and an automorphism $D\to D$}.

\begin{lem}\label{lem:uniqueness}
Let $C$ be a non-hyperelliptic curve of genus $3$ such that a subgroup $G$ of $\Aut(C)$ is isomorphic to $\Cyc_6$. Then we have a unique $\Cyc_6$-elliptic quotient of $C$.
\end{lem}
\begin{proof}
We have shown the existence.
The uniqueness follows from the classification of automorphism groups
of non-hyperelliptic curve $C$ of genus $3$. The automorphism group $\Aut(C)$
(containing $\Cyc_6$) is isomorphic to $\Cyc_6$ or a group ${\rm G}_{48}$ of order $48$ (cf.\ \cite[Theorem 3.1]{LRRS}).
By a tedious computation, we can check that the set of involutions in
subgroups $G$ with $G \cong \Cyc_6$:
\[
\{h^3 \mid h \in \Aut(C),\ \#\langle h \rangle = 6\}
\]
is a singleton, see \cite[Section 6]{Top} for an explicit description
of ${\rm G}_{48}$.
\end{proof}

{
Let $\alpha \neq 0$ be an element of $K$ such that $\alpha^{-6} = r^2/4-1$.
Let $\iota_r$ be the isomorphism $E_r\simeq D$ over $K$ sending $(X,Y) \to (\alpha^2X,\alpha^3Y)$.}
We give a necessary and sufficient condition that ${\rm Aut}(C_r)$ is the cyclic group of order $6$:
\begin{prop}\label{6or48}
Assume that $r \neq \pm 2$. Then $\Aut(C_r) \cong \Cyc_6$ for $r\ne 0$, and $\Aut(C_0) \cong {\rm G}_{48}$.
\end{prop}
\begin{proof}
Let {$\iota_r\circ\rho_r: C_r \to  E_r\simeq D$} be the $\Cyc_6$-elliptic quotient
as constructed in \eqref{standard elliptic quotient}.
By using Lemma \ref{lem:uniqueness}, 
we have a canonical homomorphism $\Aut(C_r) \to \Aut({D})$.
The kernel is of order $2$, and the image, say $H$, consists of elements of $\Aut({D})$ stabilizing the set of the {branch} points of $\rho_r$. {The set of branch points considered in $E_r$ is $B:=\{O, (1,r/2),(\zeta,r/2),(\zeta^2,r/2)\}$.} 
{
Consider the exact sequence
\[
\begin{CD}
0 @>>> \Aut(E_r) @>\subset>> \Aut(D) @>\varphi>> E_r @>>> 0.
\end{CD}
\]
Here $\Aut(E_r)$ is the automorphism group of $E_r$ as an elliptic curve,
which is considered as a subgroup of $\Aut(D)$ via $\iota_r$,
and $\varphi$ sends $g\in\Aut(D)$ to $\iota_r^{-1}\circ g \circ \iota_r (O)$.
Since $H$ stabilizes the branch points, we have $\varphi(H) \subset B$.
If $r\ne 0$, then $\varphi(H) =\{O\}$, since $\{O\}$
is the unique subgroup of $E_r$ contained in $B$.
If $r=0$, we have $\varphi(H) = B$, since $B$ is the $2$-torsion
subgroup of $E_r$ and any element of $B$ can be lifted to an element of $H$.
As the $j$-invariant of $E_r$ is $0$, we have $\#\Aut(E_r)=6$.
If $r\ne 0$, then $H\cap \Aut(E_r)$ contains $(X,Y) \mapsto (\zeta^iX,Y)$ for $i=0,1,2$ but does not contain the hyperelliptic involution $(X,Y)\mapsto (X,-Y)$, whence $\#(H\cap \Aut(E_r))=3$. If $r=0$,
then any element of $\Aut(E_r)$ stabilizes $B$, whence
$\#(H\cap \Aut(E_r)) = \#\Aut(E_r)=6$. Thus}
the group $H$ is isomorphic to $\Cyc_3$ if $r\ne 0$, and is of order $24$ if $r=0$.
\end{proof}

{To obtain Theorem \ref{thm:main10}, it remains to show its latter part:}
\begin{quote}
For $r,r' \notin \{0,\pm2\}$, we have $C_r \cong C_{r'}$ if and only if $r^2 = {r'}^2$.
Moreover, if $C_r$ is superspecial, then $r^2 \in \mathbb{F}_{p^2}$.
\end{quote}
\begin{proof}
The ``if"-part {of the first statement} is clear: If $r=-r'$, then the map $(x,y,z) \mapsto (x,y,\sqrt{-1}z)$ defines an isomorphism $C_r \rightarrow C_{r'}$.
Let us show the ``only if"-part.
By Lemma \ref{lem:uniqueness}, the isomorphism class of $C_r$ is determined by the {branch} points {of $\iota_r\circ\rho_r: C_r \to E_r \simeq D$, up to $\Aut(D)$}. The {branch} points are
the origin $O$ and the three points on $D: Y^2 = X^3+1$ defined by 
$Y := \alpha^3r/2$ and $X^3 := 4/(r^2-4)$.
{Hence, if $C_r \simeq C_{r'}$, then
we have $r^2={r'}^2$ looking at the third power of the $X$-coordinates of the branch points, since
 $\Aut(D)$ stabilizes $X^3$ (or equivalently $Y^2$).}

Next, we assume that $C_r$ with $r \neq \pm 2$ is superspecial. It is known that any superspecial curve is defined over $\mathbb{F}_{p^2}$, and hence there exists a curve {$C'$} over $\mathbb{F}_{p^2}$ such that $C_r \simeq {C'} \otimes_{\hspace{0.3mm}\mathbb{F}_{p^2}\hspace{-0.7mm}} K$. Let ${C_r}^{\!(\sigma)}$ be the fiber product $C_r\otimes_{K,\sigma} K$, where $\sigma$ is the $p^2$-power map from $K$ to $K$.
Then we have $C_r \cong {C_r}^{\!(\sigma)}$. On the other hand ${C_r}^{\!(\sigma)} \cong C_{\sigma(r)}$ holds clearly, hence we obtain $C_r \cong C_{\sigma(r)}$. Since $r^2$ is an invariant of $C_r$,
we have $r^2 = \sigma(r^2)$. This means that $r^2 \in \mathbb{F}_{p^2}$, so the proof is completed.
\end{proof}
{
\begin{rem}
    The authors learned from a referee that it had been shown in the proof of \cite[Theorem 3.3]{LRRS} that $r^2$ is a complete invariant of $C_r$ over an algebraically closed field.
\end{rem}
}

\subsection{Correspondence to Moonen's 9th special family}\label{Moonen}

We use the same notation as in the previous subsection, and let $r \in K$ with $r\neq \pm 2$.
Since $C_r$ is a $6$-cyclic cover of the projective line $\mathbb{P}^1$ via the quotient map $C_r \to C_r/\langle \gamma \rangle \cong \mathbb{P}^1$, it is isomorphic to
\[
    M_{(t_1,t_2,t_3,t_4)}: y^6 = (x-t_1)(x-t_2)^3(x-t_3)^4(x-t_4)^4 
\]
for some pairwise distinct elements $t_1$, $t_2$, $t_3$, and $t_4$ in $K$, see Remark \ref{rem:cyclic} below where the isomorphism $C_r \cong M_{(t_1,t_2,t_3,t_4)}$ can be proved without constructing any explicit isomorphism.
Conversely, one can check that $M_{(t_1,t_2,t_3,t_4)}$ is a non-hyperelliptic curve of genus $3$ whose automorphism group contains $\Cyc_6$. 
The curve $M_{(t_1,t_2,t_3,t_4)}$ is known as {\it Moonen’s 9th special family}~\cite[Table 1]{Moonen} (this is the reason why we use the notation `$M$' for the curve).

In the following, we construct an explicit isomorphism between $C_r$ and $M_{(t_1,t_2,t_3,t_4)}$, by which we can compute the equation for $C_r$ isomorphic to a given $M_{(t_1,t_2,t_3,t_4)}$, and vise versa.
With this isomorphism, we also obtain analogues of Proposition \ref{6or48} and Theorem A for $M_{(t_1,t_2,t_3,t_4)}$.


\begin{lem}\label{lem:isom}
The curve $M_{(t_1,t_2,t_3,t_4)}$ is birational to the curve defined by the equation
\begin{equation}\label{Cs}
    \widetilde{C}_{s}: Y^3 = (X^2-1)(X^2-s)\ \,{\rm with}\ s := \frac{t_4-t_2}{t_4-t_1} \cdot \frac{t_3-t_1}{t_3-t_2}.
\end{equation}
Note that $s$ is an element of $K$ different from 0 or 1.
\end{lem}
\begin{proof}
First, we claim that the curve $M_{(t_1,t_2,t_3,t_4)}$ is isomorphic to the curve defined by
\[
    D_s: Y^6 = X^3(X-1)^4(X-s)^4.
\]
Indeed, the transformation
\[
    x \mapsto \frac{x-t_2}{x-t_1} \cdot \frac{t_3-t_1}{t_3-t_2}
\]
gives an isomorphism such that $t_1 \mapsto \infty,\ t_2 \mapsto 0,\ t_3 \mapsto 1$ and $t_4 \mapsto s$. There exists a rational map
\[
    \widetilde{C}_{s} \rightarrow D_s\,;\,(X,Y) \mapsto (X^2,XY^2),
\]
whose inverse is given by
\[
    D_s \rightarrow \widetilde{C}_{s}\,;\,(X,Y) \mapsto \biggl(\frac{Y^3}{X(X-1)^2(X-s)^2},\,\frac{Y^2}{X(X-1)(X-s)}\biggr).
\]
Then, the curve $\widetilde{C}_{s}$ is birational to $D_s$, and hence this lemma is true. 
\end{proof}

\begin{prop}\label{prop:isom}
The curve $\widetilde{C}_{s}$ is isomorphic to the curve $C_r$ with $r^2 = s + \frac{1}{s} + 2$.
\end{prop}
\begin{proof}
Let $\hspace{-0.5mm}\sqrt[3]{s}$ and $\hspace{-0.5mm}\sqrt[4]{s}$ be a third root and a fourth root of $s$, and we denote by $\hspace{-0.5mm}\sqrt{s} := (\hspace{-0.6mm}\sqrt[4]{s}\hspace{0.2mm})^2$.
Then, by setting $x := -Y/\hspace{-0.5mm}\sqrt[3]{s},\ y := X/\hspace{-0.5mm}\sqrt[4]{s}$ and $z := 1$, the curve $\widetilde{C}_{s}$ is isomorphic to
\[
    C_{r}: x^3z + y^4 + ry^2z^2 + z^4 = 0\ \,{\rm with}\ \,r := -\frac{s+1}{\sqrt{s}},
\]
as desired.
\end{proof}

By Lemma \ref{lem:isom} and Proposition \ref{prop:isom}, one can easily compute a value of $r$ for which $C_r$ is isomorphic to $M_{(t_1,t_2,t_3,t_4)}$.
Conversely, for a given $r$ with $r\neq   \pm 2$, we may choose three among $t_1$, $t_2$, $t_3$, and $t_4$ arbitrary, and then we can compute the other one such that $C_r\cong M_{(t_1,t_2,t_3,t_4)}$.
For example, when $t_1$, $t_2$, and $t_3$ are chosen, one may set $t_4: = \frac{t_3 (t_2 - t_1) s  - t_1 (t_2 - t_3) }{ (t_2 - t_1) s  - (t_2 - t_3)  }$, where $s$ is taken to be any element with $s^2+(2-r^2)s+1=0$.


Let $M_{(t_1,t_2,t_3,t_4)}$ and $M_{({t'}_{\!1},{t'}_{\!2},{t'}_{\!3},{t'}_{\!4})}$ be curves belonging to Moonen’s 9th special family, and put
\begin{align*}
    A &:= (t_3-t_1)(t_4-t_2) + (t_3-t_2)(t_4-t_1),\\
    B &:= (t_3-t_1)(t_3-t_2)(t_4-t_1)(t_4-t_2),\\
    A' &:= ({t'}_{\!3}-{t'}_{\!1})({t'}_{\!4}-{t'}_{\!2}) + ({t'}_{\!3}-{t'}_{\!2})({t'}_{\!4}-{t'}_{\!1}),\ \text{and}\\
    B' &:= ({t'}_{\!3}-{t'}_{\!1})({t'}_{\!3}-{t'}_{\!2})({t'}_{\!4}-{t'}_{\!1})({t'}_{\!4}-{t'}_{\!2}).
\end{align*}
Note that $B$ and $B'$ are non-zero values since $t_i \neq t_j$ and ${t'}_{\!i} \neq {t'}_{\!j}$ if $i \neq j$.
Furthermore, let $C_r$ and $C_{r'}$ be curves constructed in Lemma \ref{lem:isom} and Proposition \ref{prop:isom} so that they are isomorphic to $M_{(t_1,t_2,t_3,t_4)}$ and $M_{({t'}_{\!1},{t'}_{\!2},{t'}_{\!3},{t'}_{\!4})}$ respectively.
Here, we obtain Corollary \ref{analogue1} and Corollary \ref{analogue2} below which are analogues of Proposition \ref{6or48} and Theorem A.
\begin{cor}\label{analogue1}
With the notation as above, the automorphism group of $M_{(t_1,t_2,t_3,t_4)}$ is isomorphic to $\Cyc_6$ if and only if $A \neq 0$. Otherwise, it is isomorphic to ${\rm G}_{48}$.
\end{cor}
\begin{proof}
Recall from Proposition \ref{6or48} that the automorphism group of $C_r$ is of order 48 if and only if $r = 0$. This is equivalent to $s = -1$ since $r^2 = s+\frac{1}{s}+2$, which is equivalent to
\[
    (t_3-t_1)(t_4-t_2) + (t_3-t_2)(t_4-t_1) = 0 
\]
by using the definition of $s$ in \eqref{Cs}.
Therefore, the statement is true.
\end{proof}
\begin{cor}\label{analogue2}
With the notation as above, we assume that $A,A' \neq 0$.
Then $M_{(t_1,t_2,t_3,t_4)}$ is isomorphic to $ M_{({t'}_{\!1},{t'}_{\!2},{t'}_{\!3},{t'}_{\!4})}$ if and only if $A^2/B = {A'}^2/B'$.
\end{cor}
\begin{proof}
Recall from Theorem A that  $C_r \cong C_{r'}$ if and only if $r^2 = r'^2$.
We have
\begin{align*}
    r^2 = s + \frac{1}{s} + 2 &= \frac{t_4-t_2}{t_4-t_1} \cdot \frac{t_3-t_1}{t_3-t_2} + \frac{t_4-t_1}{t_4-t_2} \cdot \frac{t_3-t_2}{t_3-t_1} + 2\\
    &= \frac{\{(t_3-t_1)(t_4-t_2)+(t_3-t_2)(t_4-t_1)\}^2}{(t_3-t_1)(t_3-t_2)(t_4-t_1)(t_4-t_2)} = \frac{A^2}{B}
\end{align*}
and ${r'}^2 = {A'}^2/B'$ similarly, so that the assertion is true.
\end{proof}

\begin{rem}\label{rem:cyclic}
We can prove that any $C_r$ is isomorphic to {some curve} $M_{(t_1,t_2,t_3,t_4)}$, as follows:
By considering the {\it type} of $C_r$ as a cyclic cover of $\mathbb{P}^1$ and its {\it Nielsen classes} as in \cite{Wangyu} for the case of characteristic zero, a straightforward computation of ramification indices with Hurwitz's formula implies that $C_r$ is isomorphic to a cyclic cover of $\mathbb{P}^1$ of type $(6;1,3,4,4)$ or $(6;1,3,3,5)$.
Here, the latter case is impossible.
Indeed, as in \cite[Remark 4]{WangyuSakai}, one can verify that the quotient map from the curve $y^6 = (x-t_1) (x-t_2)^3 (x - t_3)^3 (x-t_4)^5$ to its quotient by $(x,y) \mapsto (x,-y)$ defines the hyperelliptic structure (namely a morphism to $\mathbb{P}^1$ of degree two), which contradicts that $C_r$ is non-hyperelliptic.

As a method to determine the possible $a$-numbers of a cyclic cover $C$ of $\mathbb{P}^1$, we can apply Elkin's result~\cite{Elkin} on the rank of the Cartier operator $\mathscr{C}$ on the space $H^0 (C,\Omega_{C}^1)$ of regular differentials, which is dual to the Frobenius operator on $H^1(C,\mathcal{O}_C)$.
Indeed, applying his result to $C:=M_{(t_1,t_2,t_3,t_4)}$ {yields} $a(C) \geq 1$ for $p \equiv 5 \pmod{6}$, but could not imply \textit{any more}:
Letting $\zeta$ be a primitive $6$-th root of unity and denoting by $\delta$ the automorphism $(x,y)\to (x,\zeta^{-1}y)$ on $C$, we can decompose the linear space $H^0 (C,\Omega_{C}^1)$ into a direct sum of $\zeta^i$-eigenspaces $D_i$ of the induced automorphism $\delta^{\ast}$ on $H^0 (C,\Omega_{C}^1)$.
The dimension $d_i$ of each $D_i$ is computed as $(d_0,d_1,d_2,d_3,d_4,d_5) = (0,1,0,0,1,1)$, and moreover we have (i) $\mathscr{C}(D_5) \subset D_1$, $\mathscr{C} (D_4) \subset D_2 = \{0 \}$, and $\mathscr{C}(D_1) \subset D_5$, if $p \equiv 5 \pmod{6}$; and (ii) $\mathscr{C} (D_i) \subset D_i$ for any $i$, if $p \equiv 1\pmod{6}$.
Thus it follows from (i) that $\mathrm{rank}(\mathscr{C})\leq 2$, i.e., $a(C) \geq 1$ for $p \equiv 5 \pmod{6}$.
This will be also examined in Lemma \ref{C6HW} below, where we find that many entries of the Hasse-Witt matrix of $C_r$ are zero.
It is required for the complete determination of $a(C)$ to analyze ``the possibly non-zero entries'' (namely $c_1(r), c_2(r)$ in Lemma \ref{C6HW}, and $\tilde{c}_1(r), \tilde{c}_2(r), \tilde{c}_3(r)$ in Lemma \ref{C6HW2}) in the Hasse-Witt matrix, and we will do it in Subsections \ref{subsec:main1} and \ref{subsec:main2} below.

\end{rem}

\section{Explicit computation of Hasse-Witt matrices and $a$-numbers of our curves}\label{sec:C6}
Let $K$ be an algebraically closed field of characteristic $p \geq 5$.
In this section, we determine the $a$-numbers of non-hyperelliptic curves of genus $3$ with cyclic automorphism groups of order $6$.
This is not a new result (see below for details), but we give another proof of it. As a new result, at the {end} of this section, we give the exact number of isomorphism classes of such curves attaining the possible maximal $a$-number (Theorem B). Recall from Section \ref{sec:Intro} that our curves are defined by
\[
	C_r: x^3z + y^4 + ry^2z^2 + z^4 = 0, \quad {\rm for}\ r \in K\ {\rm with}\ \,r \neq 0,\pm 2.
\]
From this, unless otherwise noted, we assume $r \neq 0, \pm 2$ until the end of this section, and we define the polynomial $F := x^3z + y^4 + ry^2z^2 + z^4 \in K[x,y,z]$.


\subsection{The possible $a$-numbers of $C_r$ from Li-Mantovan-Pries-Tang’s works}\label{LMPT}

Li-Mantovan-Pries-Tang \cite[Section 6]{LMPT1} determine the possible Newton polygons of Moonen's special families.
As studied in Subsection \ref{Moonen}, Moonen's 9th special family corresponds to our curves, and it implies that the possible Newton polygons (and $p$-ranks) of our curves are determined as follows:
\begin{table}[H]\label{tab:Z_6}
 \centering
  \begin{tabular}{|c|c|c|c|}
   \hline
    & $a$-number & $p$-rank & Newton polygon\\\hline
$p\equiv 1 \pmod 6$ & $0$ & $3$ & $3(1,0)+3(0,1)$\\
\cline{2-4}
& $2$ & $1$ & $(1,0)+2(1,1)+(0,1)$\\
\hline
$p\equiv 5 \pmod 6$ & $1$ & $2$ & $2(1,0)+(1,1)+2(0,1)$\\
\cline{2-4}
& $3$ & $0$ & $3(1,1)$\\
\hline
  \end{tabular}
 \caption{The $a$-number, $p$-rank and Newton polygon of $C_r$}
\end{table}

In fact, we also determine the possible $a$-numbers of $C_r$ as follows: We denote by $f_{C_r}$ the $p$-rank of $C_r$, which is equal to the rank of $H \cdot H^{(p)} \cdots H^{(p^{g-1})}$, where $H^{(p^i)}$ is the matrix with entries equal to the $p^i$-th powers of the entries of the Hasse-Witt matrix $H$ of $C_r$. Here, the Hasse-Witt matrix $H$ of $C_r$ with respect to the ordered basis $\bigl\{\frac{1}{x^2 y z}, \frac{1}{x y^2 z} , \frac{1}{x y z^2} \bigr\}$ for $H^1(C_r,\mathcal{O}_{C_r})$ is given as
\begin{equation}\label{eq:matrix_genus3}
H=
\left(
\begin{array}{ccc}
c_{2p-2,p-1,p-1} & c_{p-2,2p-1,p-1} & c_{p-2,p-1,2p-1} \\
c_{2p-1,p-2,p-1} & c_{p-1,2p-2,p-1} & c_{p-1,p-2,2p-1} \\
c_{2p-1,p-1,p-2} & c_{p-1,2p-1,p-2} & c_{p-1,p-1,2p-2}
\end{array}
\right)
\end{equation}
where $c_{i,j,k}$ denotes the coefficient of $x^i y^j z^k$ in $F^{p-1}$.
Moreover one can check that $H$ is diagonal or anti-diagonal (see also Lemmas \ref{C6HW} and \ref{C6HW2} below). Therefore, we see that the $a$-number of $C_r$ is equal to $3-f_{C_r}$ by the following lemma:
\begin{lem}\label{lem:HWdia}
Assume that the Hasse-Witt matrix $H=(h_{i,j})_{i,j}$ of a genus-$g$ curve $C$ is diagonal, or anti-diagonal with $h_{i,g-i+1} = h_{g-i+1,i}=0$, or $h_{i,g-i+1} \neq 0$ and $h_{g-i+1,i} \neq 0$ for each $1 \leq i \leq g$. Then, we have $\mathrm{rank}(H)=f_C$.
\end{lem}
\begin{proof}
The assertion is clear in the diagonal case, and thus we here consider the anti-diagonal case.
Putting $M = H \cdot H^{(p)} \cdots H^{(p^{g-1})} = (m_{i,j})_{i,j}$, a straightforward computation implies the following:
\begin{itemize}
\item If $g$ is odd, then $M$ is anti-diagonal, and its anti-diagonal components are given by
\[
    m_{i,g-i+1} = (h_{g-i+1,i})^{1+p^2+\cdots +p^{g-1}} (h_{i,g-i+1})^{p+p^3+\cdots+p^{g-2}}
\]
for all $1 \leq i \leq g$.
\item If $g$ is even, then $M$ is diagonal, and its diagonal components are given by
\[
    m_{i,i} = (h_{i,g-i+1})^{1+p^2+\cdots +p^{g-2}} (h_{g-i+1,i})^{p+p^3+\cdots+p^{g-1}}
\]
for all $1 \leq i \leq g$.
\end{itemize}
Thus, the rank of $H$ is equal to that of $M$.
\end{proof}

Summarizing the above discussion, we can complete the whole of Table 1.
Namely, the possible $a$-numbers of $C_r$ are also determined. In addition, Li-Mantovan-Pries-Tang's another result \cite[Corollary 7.2]{LMPT2} implies that there is a smooth $C_r$ attaining the possible maximal $a$-number if $p$ is sufficiently large. If $p \equiv 5 \pmod{6}$, the number of isomorphism classes of such curves are known according to \cite[Theorem 3.15, II(e)]{Brock}, but it is not known if $p \equiv 1 \pmod{6}$.

In the following two subsections, we provide another proof of the result on the possible $a$-numbers of $C_r$ by a different way from their works.
Specifically, we compute the Hasse-Witt matrix of $C_r$ explicitly by using a Gauss' hypergeometric series.
Our method also enables us to compute the exact number of isomorphism classes of $C_r$ attaining the possible maximal $a$-number.
This is done in two different cases; where $p \equiv 5 \pmod{6}$ and where $p \equiv 1 \pmod{6}$.
These two cases are studied respectively in Subsections \ref{subsec:main1} and \ref{subsec:main2} below.
We then obtain Theorem \ref{thm:main11} by summarizing the main results of the two subsections, which are stated in Theorems \ref{thm:main1} and \ref{thm:main2}.

\subsection{Hasse-Witt matrices and $a$-numbers in the case $p \equiv 5 \pmod{6}$}\label{subsec:main1}

Throughout this subsection, assume that $p \equiv 5 \pmod{6}$ unless otherwise noted. Let us start with proving that the Hasse-Witt matrix $H$ of $C_r$ is anti-diagonal.

\begin{lem}\label{C6HW}
The Hasse-Witt matrix $H$ of $C_r$ with respect to the ordered basis $\bigl\{\frac{1}{x^2 y z}, \frac{1}{x y^2 z} , \frac{1}{x y z^2} \bigr\}$ for $H^1(C_r,\mathcal{O}_{C_r})$ is given as follows:
\setcounter{equation}{1}
\begin{equation}\label{HW}
H = \left( \begin{array}{ccc}
	0 & 0 & c_1(r)\\
	0 & 0 & 0 \\
	c_2(r) & 0 & 0 \\ 
	\end{array} \right),
\end{equation}
where $c_1(r)$ and $c_2(r)$ respectively denote the coefficients of $x^{p-2}y^{p-1}z^{2p-1}$ and $x^{2p-1}y^{p-1}z^{p-2}$ in $F^{p-1}$.
\end{lem}
\begin{proof}
Using the multinomial theorem, one can verify that in every monomial of $F^{p-1}$, the exponent of $x$ is a multiple of $3$ and that of $y$ is even. Hence, the coefficients $c_{i,j,k}$ of \eqref{eq:matrix_genus3} are automatically zero, except for the case where $(i,j,k) = (2p-1,p-1,p-2)$ and $(p-2,p-1,2p-1)$.
\end{proof}

In the following, we investigate properties of $c_1(r)$ and $c_2(r)$ as polynomials in $r$, by which we can determine the $a$-number of $C_r$ in $p \equiv 5 \pmod{6}$.
In particular, we shall prove in Proposition \ref{c2divc1} (3) below that $c_1(r)$ is divisible by $c_2(r)$.
In the proof, the formula~\eqref{eq:Gauss} below on {\it Gauss' hypergeometric series plays a key role.
This formula also enables us to obtain certain equalities (\eqref{eq:G1} and \eqref{eq:G2} below) that correlate $c_1(r)$ and $c_2(r)$ with Gauss' hypergeometric series:
The equality \eqref{eq:G2} will be used to prove the separability of $c_2(r)$ in Lemma \ref{c2sing} (3) below, from which we deduce our main results in Theorem \ref{thm:main1} below.
Here, recall that a Gauss' hypergeometric series is defined to be
\[
    G(a,b,c\,|\,t) := \sum_{n=0}^\infty \frac{(a\,;n)(b\,;n)}{(c\,;n)(1\,;n)}t^n, \quad a,b,c \in \mathbb{C},\ -c \notin \mathbb{N},
\]
where $(x\,;n) = \prod_{j=1}^n(x+j-1)$ denotes the Pochhammer symbol.}

{
\begin{prop}\label{Bro}
Let $a$ and $b$ be integers with $0 \leq b \leq a$. Then, we have
\begin{equation}\label{eq:Gauss}
        \sum_{n=0}^\infty \binom{a}{b-n}\binom{a}{n}t^n = \binom{a}{b} G(-a,-b,1+a-b\,|\,t)
\end{equation}
as formal series in $t$.
\end{prop}
\begin{proof}
For all integers $n \geq 0$, the $t^n$-coefficient of right-hand side is
\begin{align*}
    \binom{a}{b}\frac{(-a\,;n)(-b\,;n)}{(1+a-b\,;n)(1\,;n)} &= \binom{a}{b}\frac{(-a\,;n)}{(1\,;n)}\frac{(-b\,;n)}{(1+a-b\,;n)}\\
    &= \binom{a}{b} \cdot (-1)^n\binom{a}{n} \cdot \frac{(-b)(-b+1) \cdots (b+n-1)}{(a-b+1) \cdots (a-b+n)}\\
    & = \binom{a}{n}\binom{a}{b}\frac{b(b-1) \cdots (b-n+1)}{(a-b+1) \cdots (a-b+n)}\\
    & = \binom{a}{n}\frac{a!}{b!(a-b)!}\frac{b!}{(b-n)!}\frac{(a-b)!}{(a-b+n)!}\\
    &= \binom{a}{n}\frac{a!}{(b-n)!(a-b+n)!} = \binom{a}{n}\binom{a}{b-n},
\end{align*}
which is equal to the $t^n$-coefficient of left-hand side.
\end{proof}
}

Here, we regard $c_1(r)$ and $c_2(r)$ as symmetric polynomials in $\alpha$ and $\beta$ via the relations $\alpha + \beta = r$ and $\alpha \beta=1$.
Setting $z=1$, we then have
\[
	F^{p-1} = (x^3 + y^4 + ry^2 + 1)^{p-1} = \sum_{i=0}^{p-1} \binom{p-1}{i} x^{3i}(y^4 + ry^2 + 1)^{p-1-i}.
\]
Hence, we have
\begin{align}\label{d1c1d2c2}
    \begin{split}
    c_1(r) &= \binom{p-1}{(p-2)/3}d_1(r),\\
    c_2(r) &= \binom{p-1}{(2p-1)/3}d_2(r) = \binom{p-1}{(p-2)/3}d_2(r),
    \end{split}
\end{align}

where $d_1(r)$ denotes the coefficient of $y^{p-1}$ in $(y^4+ry^2+1)^{(2p-1)/3}$ and $d_2(r)$ denotes the coefficient of $y^{p-1}$ in $(y^4+ry^2+1)^{(p-2)/3}$.

\begin{prop}\label{c2divc1}
With notation as above, we have the following:
\begin{enumerate}
\item If $p \equiv 5 \pmod{12}$, the polynomial $d_2(r)$ is a polynomial in $r^2$ with a non-zero constant term. Otherwise, that is, if $p \equiv 11 \pmod{12}$, the polynomial $d_2(r)$ can be divided by $r$, and $d_2(r)/r$ is a polynomial in $r^2$ with a non-zero constant term.
\item Putting $\alpha + \beta = r$ and $\alpha \beta=1$, the polynomials $d_1(r)$ and $d_2(r)$ can be written as symmetric polynomials in $\alpha$ and $\beta$ as follows:
\begin{align}
	d_1(r) &= (\alpha\beta)^{(p+1)/6}\sum_{i+j =(p-1)/2}\binom{(2p-1)/3}{i}\binom{(2p-1)/3}{j}\alpha^i\beta^j,\label{d1deg1} \\
	d_2(r) &= \sum_{i+j =(p-5)/6}\binom{(p-2)/3}{i}\binom{(p-2)/3}{j}\alpha^i\beta^j,\label{d2deg}
\end{align}
where $(\alpha \beta)^{-(p+1)/6} d_1(r) $ and $d_2(r)$ are not divided by $\alpha \beta$ as polynomials in $\alpha$ and $\beta$.
Hence, we have $\mathrm{deg}_r (c_1) = \mathrm{deg}_r (d_1) = (p-1)/2$ and $\mathrm{deg}_r (c_2) = \mathrm{deg}_r (d_2) = (p-5)/6$.
\item The polynomial $d_1(r)$ is divided by the polynomial $d_2(r)$ modulo $p$.
\end{enumerate}
\begin{proof}
(1) It follows from the multinomial theorem that
\[
(y^4+ry^2+1)^{(p-2)/3} = 
\sum_{a + b + c = (p-2)/3} \binom{p-1}{a,b,c} r^b y^{4a + 2 b}
\]
whose $y^{p-1}$-coefficient is equal to
\[
d_2(r) = \sum_{\substack{4a + 2b = p-1\\a+b+c=(p-2)/3}}\binom{p-1}{a,b,c} r^{b}.
\]
If $p \equiv 5 \pmod{12}$, then $p-1$ is a multiple of $4$, hence an integer $b$ satisfying $4a + 2b=p-1$ for some integer $a$ is even. Moreover, the constant term of $d_2(r)$ is equal to $\binom{p-1}{a,b,c}$ with $a = (p-1)/4,\,b = 0$ and $c = (p-5)/12$. This is clearly not zero, so statement (1) in the case that $p \equiv 5 \pmod{12}$ is true. If $p \equiv 11 \pmod{12}$, then $p-1$ is not a multiple of $4$, and thus an integer $b$ satisfying $4a + 2b=p-1$ for some integer $a$ is odd, and hence $d_2(r)$ is divided by $r$. Moreover, the constant term of $d_2(r)/r$ is equal to $\binom{p-1}{a,b,c}$ with $a = (p-3)/4,\,b = 1$ and $c = (p-11)/12$, and this is also non-zero.

(2) Setting $y^4 + ry^2 + 1 = (y^2+\alpha)(y^2+\beta)$ with $\alpha,\beta \in K$, we obtain the binomial expansions
\begin{align*}
    &(y^2+\alpha)^{(2p-1)/3}(y^2+\beta)^{(2p-1)/3}\\
    &= \sum_{k=0}^{(4p-2)/3}\sum_{i+j=k}\binom{(2p-1)/3}{i}\binom{(2p-1)/3}{j}\alpha^i\beta^jy^{(8p-4)/3-2k},
\end{align*}
and
\begin{align*}
    &(y^2+\alpha)^{(p-2)/3}(y^2+\beta)^{(p-2)/3}\\
    &= \sum_{k=0}^{(2p-4)/3}\sum_{i+j=k}\binom{(p-2)/3}{i}\binom{(p-2)/3}{j}\alpha^i\beta^jy^{(4p-8)/3-2k}.
\end{align*}
Looking at the $y^{p-1}$-coefficients in the right hand sides of these two equalities, we have
\begin{align}\label{d1deg2}
	d_1(r) &= \sum_{i+j =(5p-1)/6}\binom{(2p-1)/3}{i}\binom{(2p-1)/3}{j}\alpha^i\beta^j,\\
    d_2(r) &= \sum_{i+j =(p-5)/6}\binom{(p-2)/3}{i}\binom{(p-2)/3}{j}\alpha^i\beta^j.\notag
\end{align}
Since $i,j$ with $i+j = (5p-1)/6$ and $i,j \leq (2p-1)/3$ must satisfy $i,j \geq (p+1)/6$, we can rewrite the right hand side of \eqref{d1deg2} as that of \eqref{d1deg1}, as desired.

(3) {Let $G^{(d)}(a,b,c \,|\, t)$ be the truncation of Gauss' hypergeometric series $G(a,b,c \,|\, t)$ to degree $d$ modulo $p$, that is, we define
\[
    G^{(d)}(a,b,c\,|\,t) := \sum_{n=0}^d g_nt^n, \quad g_n := \frac{(a\,;n)(b\,;n)}{(c\,;n)(1\,;n)} \ \ \text{mod}\ p
\]
where $(x\,;n) = \prod_{j=1}^n(x+j-1)$.
It follows from \eqref{eq:Gauss} of Proposition \ref{Bro} that
\begin{align*}
    &\sum_{i+j=(p-1)/2}\binom{(2p-1)/3}{i}\binom{(2p-1)/3}{j}t^i\\
    &= \binom{(2p-1)/3}{(p-1)/2}G^{(p-1)/2}
    \left(\frac{1-2p}{3},\frac{1-p}{2},\frac{p+7}{6}\ \middle|\ t\right)\\
    &\equiv \binom{(2p-1)/3}{(p-1)/2}G^{(p-1)/2}
    (1/3,1/2,7/6 \mid t) \pmod{p}
\end{align*}
and
\begin{align*}
    &\sum_{i+j =(p-5)/6}\binom{(p-2)/3}{i}\binom{(p-2)/3}{j}t^i\\
    &= \binom{(p-2)/3}{(p+1)/6}
    \left(\frac{2-p}{3},\frac{5-p}{6},\frac{p+7}{6}\ \middle|\ t\right)\\
    &\equiv \binom{(p-2)/3}{(p+1)/6}G^{(p-5)/6}(2/3,5/6,7/6 \mid t) \pmod{p}.
\end{align*}
Substituting $t = \alpha/\beta$ and multiplying $\beta^{(p-1)/2}$ or $\beta^{(p-5)/6}$ of both sides, we have
\begin{align*}
    &\sum_{i+j=(p-1)/2}\binom{(2p-1)/3}{i}\binom{(2p-1)/3}{j}\alpha^i\beta^j\\
    &\equiv \beta^{(p-1)/2}\binom{(2p-1)/3}{(p-1)/2}G^{(p-1)/2}
    (1/3,1/2,7/6 \mid \alpha/\beta),
\end{align*}
and
\begin{align*}
    &\sum_{i+j =(p-5)/6}\binom{(p-2)/3}{i}\binom{(p-2)/3}{j}\alpha^i\beta^j\\
    &\equiv \beta^{(p-5)/6}\binom{(p-2)/3}{(p+1)/6}G^{(p-5)/6}(2/3,5/6,7/6 \mid \alpha/\beta).
\end{align*}
Then, it follows from \eqref{d1deg1} and \eqref{d2deg} that
\begin{align}
    \begin{split}
    &\binom{(2p-1)/3}{(p-1)/2}\beta^{(p-1)/2}G^{(p-1)/2}(1/3,1/2,7/6 \mid \alpha/\beta)\\
    &\equiv (\alpha \beta)^{-(p+1)/6} d_1(r) \pmod{p}\label{eq:G1}
    \end{split}
\end{align}
and
\begin{align}
    \begin{split}
    &\binom{(p-2)/3}{(p+1)/6}\beta^{(p-5)/6}G^{(p-5)/6}(5/6,2/3,7/6 \mid \alpha/\beta)\\
    &\equiv d_2(r) \pmod{p}. \label{eq:G2}
    \end{split}
\end{align}

On the other hand, using Euler's transformation formula, we obtain
\begin{equation}\label{Euler}
	G(a,b,c\,|\,t) = (1-t)^{c-a-b}G(c-a,c-b,c\,|\,t).
\end{equation}
Setting $a = (1-2p)/3,\,b = (1-p)/2,\,c = (p+7)/6$ and $t = \alpha/\beta$, we have}
\begin{align}\label{trans}
    \begin{split}
    &G^{(p-1)/2}(1/3,1/2,7/6 \mid \alpha/\beta)\\
    &\equiv (1-\alpha/\beta)^{(p+1)/3}G^{(p-5)/6}(5/6,2/3,7/6 \mid \alpha/\beta) \pmod{p}.
    \end{split}
\end{align}
We remark that $G(1/3,1/2,7/6 \mid t)$ modulo $p$ does not have any $t^n$-term for $(p-1)/2 \leq n \leq p-1$ since {the} values $(1/2\,;n)$ for all integers $(p-1)/2 \leq n \leq p-1$ are divided by $p$. Similarly to this, {the} values $(5/6\,;n)$ for all integers $(p-5)/6 \leq n \leq p-1$ can be divided by $p$, and hence we see that the $t^n$-terms of $G(5/6,2/3,7/6 \mid t)$ modulo $p$ for $(p-5)/6 \leq n \leq p-1$ are all vanishing.
Therefore, by multiplying $\binom{(p-2)/3}{(p+1)/6}\binom{(2p-1)/3}{(p+1)/6}(\alpha\beta)^{(p+1)/6}\beta^{(p-1)/2}$ to both sides of (\ref{trans}), we get the equality
\begin{align}\label{d2divd1}
    \begin{split}
    &\binom{(p-2)/3}{(p+1)/6}d_1(r)\\
    &\equiv \binom{(2p-1)/3}{(p+1)/6}(\alpha\beta)^{(p+1)/6}(\beta-\alpha)^{(p+1)/3}d_2(r) \pmod{p}
    \end{split}
\end{align}
by \eqref{eq:G1} and \eqref{eq:G2}.
Hence $d_1(r)$ is divided by $d_2(r)$ as symmetric polynomials in $\alpha$ and $\beta$.
\end{proof}
\end{prop}


Thanks to Proposition \ref{c2divc1}, we can compute the $a$-number of $C_r$ by determining whether $c_2(r)$ is zero or not.
For this, we prove the following lemma, which collects properties of $c_2(r)$.

\begin{lem}\label{c2sing}
With notation as above, we have the following:
\begin{enumerate}
\item As a polynomial in $\alpha$ and $\beta$, the polynomial $d_2 (r)$ cannot be divided by $\beta-\alpha$.
\item The polynomial $c_2(r)$ does not have roots $r = \pm 2$.
\item The polynomial $c_2(r)$ is separable.
\end{enumerate}
\end{lem}
\begin{proof}
(1) If $\alpha = \beta$, then it follows from \eqref{d2deg} in Proposition \ref{c2divc1} (2) together with Vandermonde's convolution theorem that
\begin{align*}
    d_2(r) &= \sum_{i+j =(p-5)/6}\binom{(p-2)/3}{i}\binom{(p-2)/3}{j}\beta^{(p-5)/6}\\
    &= \beta^{(p-5)/6}\binom{(2p-4)/3}{(p-5)/6} \neq 0.
\end{align*}

(2) Assume that $c_2(r)$ has a root $r=2$ (resp.\ $r = -2$).
Since the curve $x^4 + y^4 + 2y^2z^2 + z^4 = 0$ and the curve $x^4 + y^4 - 2y^2z^2 + z^4 = 0$ are clearly isomorphic, then $c_2(r)$ has another root $r = -2$ (resp.\ $r = 2$). Hence $c_2(r)$ can be divided by $r^2-4 = (\alpha+\beta)^2 - 4\alpha\beta = (\beta-\alpha)^2$, and so does $d_2(r)$. This contradicts the statement (1).

(3) Using \eqref{d1c1d2c2}, it suffices to show the separability of $d_2(r)$.
Similarly to Igusa's proof in \cite{Igusa} for the separability of the Deuring polynomial, we can prove the assertion as follows:
We first claim that 
\[
G(t) := G^{(p-5)/6}(5/6,2/3,7/6 \mid t)
\]
with $t := \alpha/\beta$ is a separable polynomial.
Indeed, it follows from the statement (1) together with \eqref{d2deg} {and \eqref{eq:G2}} that the roots $t$ of $G(t)$ are different from $0$ and $1$. 
Here, by Euler’s hypergeometric differential equation, we have that $G(a,b,c\,|\,t)$ satisfies the differential equation $\mathcal{D}G(a,b,c\,|\,t) = 0$ with
\begin{equation}\label{Gaussdif}
	\mathcal{D} := t(1-t)\frac{d^2}{dt^2} + \bigl(c-(a+b+1)t\bigr)\frac{d}{dt} -ab.
\end{equation}
Assume that $G(t)$ has a multiple root $t_0$, then we have $G(t_0) = {G}'(t_0) = 0$ and ${G}''(t_0) = 0$ by \eqref{Gaussdif}. 
Repeating this inductively, we thus have ${G}^{(n)}(t_0) = 0$ for all $n \geq 0$, where $G^{(n)}(t)$ denotes the $n$-th derivative of $G(t)$ with respect to $t$.
This is a contradiction, and so the roots of $G(t)$ are all simple.

To prove the separability of $d_2(r)$, assume for contradiction that $d_2(r)$ has a multiple root $r_0 \in K$.
Proposition \ref{c2divc1} (1) implies that $r_0 \neq 0$ and $-r_0$ is also a multiple root of $d_2(r)$. Hence $d_2(r)$ has a factor of the form $(r-r_0)^2(r+r_0)^2 = (r^2 - r_0^2)^2$. It follows from $r^2 -4 = (\beta - \alpha)^2$ and $\alpha \beta = 1$ that we obtain 
\begin{align*}
    r^2 - r_0^2 = (r^2 - 4) - (r_0^2 - 4) &= ( \beta - \alpha)^2  - (r_0^2-4) \alpha \beta \\
    &= \beta^2 \biggl\{\biggl(1-\frac{\alpha}{\beta}\biggr)^{\!\!2} - (r_0^2-4) \frac{\alpha}{\beta}\biggr\}.
\end{align*}
Recall from $t := \alpha/\beta$ that $r^2 - r_0^2 = \beta^2 \{(1-t)^2 - t(r_0^2-4)\}$, and hence we obtain
\[
    (r^2-r_0^2)^2 = \beta^4 \{t^2 + (2-r_0)t +1\}^2.
\]
Thus, we have that $G(t)$ is not separable by \eqref{eq:G2}, and hence this is a contradiction.
Therefore, any root $r_0 \in K$ of $d_2(r)$ is simple, as desired.
\end{proof}

Here, we determine when $c_2(r)$ has a root with $r \neq 0,\pm 2$, after which we also state and prove our main theorem on $C_r$ with $\mathrm{Aut}(C_r) \cong \Cyc_6$ in the case where $p \equiv 5 \pmod{6}$:

\begin{prop}\label{c2root}
The polynomial $c_2(r)$ has a root $r \in {K}$ with $r \neq 0,\pm 2$ if and only if $p \geq 17$.
\end{prop}
\begin{proof}
For $p < 17$ with $p \equiv 5 \pmod 6$, a straightforward computation shows that
\[
    \left\{
    \begin{array}{l}
        c_2(r) = -1 \quad \hspace{1.7mm}{\rm for}\ p = 5,\\
        c_2(r) = -3r \quad {\rm for}\ p = 11.
    \end{array}
    \right.
\]
Thus if $p = 5$ or $p = 11$, there does not exist $r \neq 0$ such that $c_2(r) = 0$.
For $p \geq 17$, the claim holds from Lemma \ref{c2sing}\,(2)(3) as the degree of $c_2(r)$ is equal to $(p-5)/6 \geq 2$ by Proposition \ref{c2divc1}\,(2). 
\end{proof}

\begin{theorem}\label{thm:main1}
Assume $r \neq 0, \pm 2$, and $p \equiv 5 \pmod{6}$.
Then, we have the following:
\begin{enumerate}
\item The $a$-number of the non-singular curve $C_r$ is equal to $1$ if $p < 17$, and $1$ or $3$ if $p \geq 17$.
\item The number of non-singular curves $C_r$ with $a(C_r) = 3$ is equal to $\lfloor{(p-5)/12}\rfloor$.
\end{enumerate}
\begin{proof}
(1) As the Hasse-Witt matrix of $C_r$ is given by \eqref{HW} in Lemma \ref{C6HW}, it is clear that $a(C_r) \geq 1$.
Assume that $a(C_r) = 2$.
It follows from Proposition \ref{c2divc1} (3) that $c_1(r)$ is divided by $c_2(r)$.
Then, we have $c_1(r) = 0$ and $c_2(r) \neq 0$, and thus $d_1(r)=0$ and $d_2(r) \neq 0$ by \eqref{HW} and \eqref{d1c1d2c2}.
Here, recall {that since we set} $y^4 +ry^2 + 1 = (y^2 +\alpha)(y^2 +\beta)$ {with} $\alpha\beta = 1$, then we have $(\beta-\alpha)^2 = 0$ by using the equality \eqref{d2divd1}.
By the relation $(\beta-\alpha)^2 = (\alpha + \beta)^2 - 4 \alpha \beta = r^2 - 4$, we also have $r = \pm 2$, which contradicts our assumption $r \neq \pm 2$.
The claim follows from Proposition \ref{c2root}.

(2) By Lemma \ref{C6HW} and Proposition \ref{c2divc1}, the curve $C_r$ with $r \neq 0,\pm 2$ has $a$-number $3$ if and only if $r$ is a root of $c_2(r)$ of degree $(p-5)/6$.
It follows from Lemma \ref{c2sing} (2)(3) that the number of different roots such that $r \neq \pm 2$ of $c_2(r)$ is equal to $(p-5)/6$.
Here we claim that $r=0$ is a root of $c_2(r)$ if and only if $p \equiv 11 \pmod{12}$. Indeed recall from \eqref{d1c1d2c2} that $c_2(r)$ is a constant multiple of the polynomial $d_2(r)$, which is defined as the coefficient of $y^{p-1}$ in the polynomial $(y^4+ry^2+1)^{(p-2)/3}$.
One can check that this coefficient is equal to $0$ for $r=0$ if and only if $p \equiv 11 \pmod{12}$.
Therefore, the number of different roots with $r\neq 0,\pm 2$ of $c_2(r)$ is
\[
	\left\{
	\begin{array}{l}
		(p-5)/6, \hspace{6.5mm} \quad p \equiv 5 \hspace{2mm}\pmod{12},\\
		(p-5)/6-1, \quad p \equiv 11 \pmod{12}.
	\end{array}
	\right.
\]
By Theorem \ref{thm:main10}, the number of isomorphism classes is half of this, as desired.
\end{proof}
\end{theorem}

\subsection{Hasse-Witt matrices and $a$-numbers in the case $p \equiv 1 \pmod{6}$}\label{subsec:main2}

Throughout this subsection, assume that $p \equiv 1 \pmod{6}$ unless otherwise noted.
Let us start with proving that the Hasse-Witt matrix $H$ of $C_r$ is diagonal.
\begin{lem}\label{C6HW2}
The Hasse-Witt matrix $H$ of $C_r$ with respect to the ordered basis $\bigl\{\frac{1}{x^2 y z}, \frac{1}{x y^2 z} , \frac{1}{x y z^2} \bigr\}$ for $H^1({C_r},\mathcal{O}_{{C_r}})$ is given as follows:
\begin{equation}\label{HW2}
H=
\left( \begin{array}{ccc}
	\tilde{c}_{1}(r) & 0 & 0\\
	0 & \tilde{c}_{2}(r) & 0 \\
	0 & 0 & \tilde{c}_{3}(r) \\ 
	\end{array} \right),
\end{equation}
where we denote by $\tilde{c}_{1}(r)$, $\tilde{c}_{2}(r)$ and $\tilde{c}_{3}(r)$ the coefficients of $x^{2p-2}y^{p-1}z^{p-1}$, $x^{p-1}y^{2p-2}z^{p-1}$ and $x^{p-1}y^{p-1}z^{2p-2}$ in $F^{p-1}$ respectively.
\end{lem}
\begin{proof}
This is proved in a way similar to the proof of Lemma \ref{C6HW}.
\end{proof}

Next, we shall prove results analogous to Proposition \ref{c2divc1} and Lemma \ref{c2sing} (2) in the following.
\begin{prop}\label{prop:c3'}
With notation same as in Lemma \ref{C6HW2}, we have the following:
\begin{enumerate}
\item The polynomial $\tilde{c}_3(r)$ is divided by the polynomial $\tilde{c}_1(r)$ modulo $p$.
\item The roots of the polynomial $\tilde{c}_2(r)$ are only $r = \pm 2$, hence $\tilde{c}_2(r) \neq 0$ for any other $r \in K$.
\end{enumerate}
\end{prop}

\begin{proof}
{Similarly to} the notations in Subsection \ref{subsec:main1}, we set $\tilde{d}_{1}(r)$, $\tilde{d}_{2}(r)$ and $\tilde{d}_{3}(r)$ as the coefficients of $y^{p-1}$, $y^{2p-2}$ and $y^{p-1}$ respectively, in the polynomials $(y^4+ry^2+1)^{(p-1)/3}$, $(y^4+ry^2+1)^{(2p-2)/3}$ and $(y^4+ry^2+1)^{(2p-2)/3}$.
Then one can confirm that
\begin{align}\label{d1c1d2c2d3c3}
    \begin{split}
    \tilde{c}_{1}(r) &= \binom{p-1}{(p-1)/3}\tilde{d}_{1}(r),\\
    \tilde{c}_{2}(r) &= \binom{p-1}{(p-1)/3}\tilde{d}_{2}(r), \\
    \tilde{c}_{3}(r) &= \binom{p-1}{(p-1)/3}\tilde{d}_{3}(r).
    \end{split}
\end{align}
Setting $y^4 +ry^2 + 1 = (y^2 +\alpha)(y^2 +\beta)$ with $\alpha,\beta \in K$ satisfying $\alpha + \beta = r$ and $\alpha \beta = 1$, we can write $\tilde{d}_1(r)$, $\tilde{d}_2(r)$ and $\tilde{d}_3(r)$ as symmetric polynomials in $\alpha$ and $\beta$;
\begin{align}
	\tilde{d}_{1}(r) &= \sum_{i+j =(p-1)/6}\binom{(p-1)/3}{i}\binom{(p-1)/3}{j}\alpha^i\beta^j,\label{d1deg}\\
	\tilde{d}_{2}(r) &= \sum_{i+j =(p-1)/3}\binom{(2p-2)/3}{i}\binom{(2p-2)/3}{j}\alpha^i\beta^j,\nonumber\\
	\tilde{d}_{3}(r) &= (\alpha\beta)^{(p-1)/6}\sum_{i+j =(p-1)/2}\binom{(2p-2)/3}{i}\binom{(2p-2)/3}{j}\alpha^i\beta^j\nonumber
\end{align}
and moreover
\begin{alignat}{2}
	\binom{(p-1)/3}{(p-1)/6}\beta^{(p-1)/6}G^{(p-1)/6}(1/6,1/3,5/6 \mid \alpha/\beta) &\equiv \tilde{d}_1(r) && \pmod{p},\nonumber\\
	\binom{(2p-2)/3}{(p-1)/3}\beta^{(p-1)/3}G^{(p-1)/3}(2/3,1/3,2/3 \mid\alpha/\beta) &\equiv \tilde{d}_2(r) && \pmod{p},\nonumber
\end{alignat}
\begin{align*}
    &\binom{(2p-2)/3}{(p-1)/2}\beta^{(p-1)/2}G^{(p-1)/2}(2/3,1/2,5/6 \mid \alpha/\beta)\\
    &\equiv (\alpha\beta)^{-(p-1)/6} \tilde{d}_3(r) \pmod{p} \nonumber
\end{align*}
similarly to the proof of Proposition \ref{c2divc1}.
Using (\ref{Euler}), we get the following two equalities:
\begin{align}
    \begin{split}
    \tilde{d}_{2}(r) &\equiv \binom{(2p-2)/3}{(p-1)/3}(\beta-\alpha)^{(p-1)/3}\label{d'2} \pmod{p},
    \end{split}
\end{align}
\begin{align}
    \begin{split}
    &\binom{(p-1)/3}{(p-1)/6}\tilde{d}_{3}(r)\\
    &\equiv \binom{(2p-2)/3}{(p-1)/2}(\alpha\beta)^{(p-1)/6}(\beta-\alpha)^{(p-1)/3}\tilde{d}_{1}(r)\label{d'1divd'3} \pmod{p}.
    \end{split}
\end{align}
Here \eqref{d'2} means that $\tilde{c}_{2}(r)$ does not have any roots $r \neq \pm 2$ since $(\beta-\alpha)^2 = r^2-4$.
Moreover, we see that \eqref{d'1divd'3} means that $\tilde{c}_{3}(r)$ is divided by $\tilde{c}_{1}(r)$.
\end{proof}

We here state and prove the main theorem in this subsection:
\begin{theorem}\label{thm:main2}
Assume $r \neq 0, \pm 2$, and $p \equiv 1 \pmod{6}$.
Then we have the following:
\begin{enumerate}
\item The $a$-number of the non-singular curve $C_r$ is equal to $0$ or $2$.
\item The number of non-singular curves $C_r$ with $a(C_r) = 2$ is equal to $\lfloor{(p-1)/12}\rfloor$.
\end{enumerate}
\begin{proof}
(1) It follows from Proposition \ref{prop:c3'}\,(2), we have $a(C_r) \leq 2$. Here, we suppose that $a(C_r) = 1$, i.e.\ the Hasse-Witt matrix \eqref{HW2} of $C_r$ given in Lemma \ref{C6HW2} has rank $2$.
By Proposition \ref{prop:c3'} (1)(2), the parameter $r \in K$ must satisfy $\tilde{c}_{1}(r) \neq 0,\,\tilde{c}_{2}(r) \neq 0$ and $\tilde{c}_{3}(r) = 0$ and hence $\tilde{d}_{1}(r) \neq 0,\,\tilde{d}_{2}(r) \neq 0$ and $\tilde{d}_{3}(r) = 0$, where $\tilde{d}_{1}(r),\,\tilde{d}_{2}(r)$ and $\tilde{d}_{3}(r)$ are given in the proof of Proposition \ref{prop:c3'}.
Similarly to the proof of Theorem \ref{thm:main1} (1), it follows from \eqref{d'1divd'3} that $r = \pm 2$, a contradiction.

(2) One can verify in a way similar to the proofs of Lemma \ref{c2sing}\,(2)(3) that the zeros of $\tilde{c}_{1}(r)$ are simple and that $\tilde{c}_{1}(r)$ has no root $r = \pm 2$.
Here we claim that the degree of $\tilde{c}_{1}(r)$ is $(p-1)/6$.
Indeed, the degree of $\tilde{c}_{1}(r)$ is equal to that of $\tilde{d}_{1}(r)$ by \eqref{d1c1d2c2d3c3}.
Since the degree of $\tilde{d}_{1}(r)$ in $r$ is equal to that of $\tilde{d}_{1}(r)$ in $\alpha$ and $\beta$ by $\alpha+\beta =r$ and $\alpha\beta = 1$, it is equal to $(p-1)/6$ by using \eqref{d1deg}.
Thus, the number of different roots with $r \neq 0,\pm 2$ of $\tilde{c}_{1}(r)$ is
\[
	\left\{
	\begin{array}{l}
		(p-1)/6, \hspace{6.5mm} \quad p \equiv 1 \pmod{12},\\
		(p-1)/6-1, \quad p \equiv 7 \pmod{12}.
	\end{array}
	\right.
\]
By Theorem \ref{thm:main10}, the number of isomorphism classes is half of this, as desired.
\end{proof}
\end{theorem}
\section{Concluding remarks}\label{conclusion}

In this paper, we focused on non-hyperelliptic curves $C_r$ of genus 3 with cyclic automorphism group of order 6, and gave a necessary and sufficient condition with respect to $r$ and $r'$ such that $C_r \cong C_{r'}$ (the first assertion of Theorem A). We found another proof of the possible $a$-numbers of $C_r$, and we obtained the exact number of isomorphism classes of $C_r$ attaining the possible maximal $a$-number (Theorem B). Moreover, we showed that $r^2$ belongs to $\mathbb{F}_{p^2}\hspace{-0.3mm}$ if $C_r$ is superspecial (the second assertion of Theorem A). More strongly, the following seems to be true:
\begin{expectt}\label{exp}
If $C_r$ is superspecial, then $r \in \mathbb{F}_{p^2}$.
\end{expectt}

\begin{remm}
Any superspecial curve is isomorphic (over $\overline{\mathbb{F}_p}$) to one defined over $\mathbb{F}_{p^2}$.
Hence, the above Expectation implies that $C_r$ itself provides a model over $\mathbb{F}_{p^2}$, if it is superspecial.
\end{remm}
We confirmed that the above Expectation is true for $17 \leq p < 10000$ with help of computer calculation.
More precisely, we checked that any solution $z$ of
\begin{equation}\label{sspeq_hypergeom}
G^{(p-5)/6}(5/6,2/3,7/6 \mid t)=0
\end{equation}
is a square in $\mathbb{F}_{p^2}$ for all $17 \leq p < 10000$ with $p \equiv 5 \pmod 6$ by using Magma \cite{MagmaHP}.
This assertion implies the correctness of Expectation.
Indeed, if $C_r$ is superspecial, then we have $p \equiv 5 \pmod 6$ and $c_2(r)=0$ holds as in Lemma \ref{C6HW}. This is equivalent to \eqref{sspeq_hypergeom} setting $t:=\alpha/\beta$,
where $\alpha+\beta=r$ and $\alpha\beta=1$. 
If $t =\alpha/\beta \in (\mathbb{F}_{p^2}\hspace{-0.3mm})^2$ holds,
then $\alpha/\beta = \alpha^2 = \beta^{-2} \in (\mathbb{F}_{p^2}\hspace{-0.3mm})^2$
and therefore $r=\alpha+\beta\in \mathbb{F}_{p^2}$.

\begin{remm}\label{rem:R}
Regardless of whether the above Expectation is correct or not, we find an explicit model over $\mathbb{F}_{p^2}\hspace{-0.3mm}$ of a superspecial $C_r$, by using the second assertion of Theorem A:\ 
The curve $C_r$ is isomorphic (over $\overline{\mathbb{F}_p}$) to 
\[
    x^3z + a y^4 + a y^2 z^2 + z^4 = 0
\]
via $y \mapsto \sqrt{r} y$ with $a := r^2$. Note that this model is given in \cite[Theorem 3.3]{LRRS} as another model of genus-$3$ non-hyperelliptic curve with cyclic automorphism group of order $6$.
The above Expectation means that we can take $a$ to be a square in $\mathbb{F}_{p^2}$ for $x^3 z + a y^4 + a y^2 z^2 + z^4 = 0$ to be superspecial. The authors also learned the contents of this remark from C.\,Ritzenthaler.
\end{remm}

A theoretical proof of Expectation is a future work.

\subsection*{Acknowledgments}

The authors are grateful to referees for their helpful suggestions. The authors also thank Tomoyoshi Ibukiyama and Christophe Ritzenthaler for their useful discussions on earlier version of this paper.
This work was supported by JSPS Grant-in-Aid for Young Scientists 20K14301, and JSPS Grant-in-Aid for Scientific Research\hspace{.5mm}(C) 21K03159.
This work was also supported by JST\ CREST Grant Number JPMJCR2113, Japan. 

%
%


\end{document}